\documentclass{birkjour}
\usepackage{amsmath}
\usepackage{amssymb}
\usepackage{amsfonts}
\usepackage{float}
\usepackage{graphicx}
\usepackage{subfigure}
\usepackage{cite}
\usepackage{epstopdf}
\usepackage [latin1]{inputenc}
\usepackage{mathrsfs}
\usepackage{lineno}
\usepackage{listings}
\usepackage{color}
\usepackage{xcolor}
\usepackage{colortbl}
\usepackage{booktabs}
\usepackage{setspace}
\usepackage{makecell}
\newtheorem{thm}{Theorem}[section]

\newtheorem{lem}{Lemma}[section]

\newtheorem{rem}{Remark}[section]

\theoremstyle{definition}

\theoremstyle{remark}
\numberwithin{equation}{section}
\allowdisplaybreaks[4]
\begin{document}
\title[Lower Bound on Translative Covering Density of Octahedron]
{Lower Bound on Translative Covering \\Density of Octahedron}

\author{Yiming Li}
\address{
Center for Applied Mathematics \\
Tianjin University\\
Tianjin, 300072\\
P. R. China}
\email{xiaozhuang@tju.edu.cn}

\author{Yanlu Lian*}
\address{
School of Mathematics \\
Hangzhou Normal University\\
Hangzhou, 311121\\
P. R. China}
\email{yllian@hznu.edu.cn}

\author{Miao Fu}
\address{
Center for Applied Mathematics \\
Tianjin University\\
Tianjin, 300072\\
P. R. China}
\email{miaofu@tju.edu.cn}

\author{Yuqin Zhang}
\address{
School of Mathematics \\
Tianjin University\\
Tianjin, 300072\\
P. R. China}
\email{yuqinzhang@tju.edu.cn}

\thanks{*Corresponding author.}

\begin{abstract}
In this paper, we present the first nontrivial lower bound on the translative covering density of octahedron. To this end, we show the lower bound, in any translative covering of octahedron, on the density relative to a given parallelehedron. The resulting lower bound on the translative covering density of octahedron is $1+6.6\times10^{-8}$.
\end{abstract}
\maketitle
\section{Introduction}
\quad \quad Packing and covering, as a classic subject in pure mathematics, has a long and rich history. In 1611, J. Kepler proposed his conjecture that \emph{the density of the densest ball packing should be} $\frac{\pi}{\sqrt{18}}$. In 1900, as parts of his 18th problem, D. Hilbert \cite{Hilbert} proposed the problems \emph{to determine the densest packing density of a geometric object such as a ball or a tetrahedron}. However, it was Minkowski who carried out a systematic study of packing and covering as the core of his Geometry of Numbers. During the course, many prominent mathematicians have made contributions to this subject. For example, Gauss, Lagrange, Dirichlet, Thue, Minkowski, Voronoi, Mahler, Davenport, Delone, Rogers, L. Fejes T\'{o}th, Hlawka et al.
It is worth mentioning that the sphere packing in $\mathbb{E}^8$ was recently resolved by Viazovska \cite{Viazovska} in a stunning breakthrough, and the method was then quickly extended to solve the problem in $\mathbb{E}^{24}$ \cite{Cohn}.
For more details on packings, we refer to \cite{Betke}, \cite{Brass}, \cite{G.Fejes}, \cite{Lagarias} and \cite{Zong.1}.

As a dual topic of packing, sphere covering has also been studied by many scholars. Let $\theta^c(B_n)$, $\theta^t(B_n)$ and $\theta^l(B_n)$ denote the densities of the thinnest congruent covering, the thinnest translative covering and the thinnest lattice covering of $\mathbb{E}^n$ with unit ball $B_n$, respectively. Clearly, $\theta^c(B_n)=\theta^t(B_n)$. The known exact results of $\theta^l(B_n)$ are summarized in the following table. As for the exact value of $\theta^t(B_n)$, our only knowledge is $\theta^t(B_2)=\theta^l(B_2)=\frac{2\pi}{\sqrt{27}}$, which was discovered by R. Kershner \cite{Kershner} in 1939.
In $\mathbb{E}^n$, from the works of many authors, in particular C. A. Rogers (see \cite{RogersB}), we know that
$\frac{n}{e\sqrt{e}}\ll \theta^t(B_n)\leq \theta^l(B_n)\leq c\cdot n(\log_e n)^{\frac{1}{2}\log_2 2\pi e}.$
\renewcommand{\arraystretch}{1.2}
\begin{table}[H]\label{ball}
\definecolor{mygray}{gray}{.9}
\definecolor{mypink}{rgb}{.99,.91,.95}
\definecolor{mycyan}{cmyk}{.3,0,0,0}
\centering
\label{table1}
\begin{tabular}{c c c c c}
\toprule
  $n$  & $\theta^l(\cdot)$ & \small{Author}   \\
\midrule
  \rowcolor{mygray}
  $2$ & $\frac{2\pi}{\sqrt27}$ & \small{Kershner 1939 \cite{Kershner}}   \\
  $3$ & $\frac{5\sqrt{5}\pi}{24}$ & \small{Bambah 1954 \cite{Bambah}, Barnes 1956 \cite{Barnes}, Few 1956 \cite{Few}} \\
  \rowcolor{mygray}
  $4$ & $\frac{2\pi^2}{5\sqrt{5}}$ & \small{Delone and Ry\u{s}kov 1963 \cite{Delone}}\\
  $5$ & $\frac{245\sqrt{35}\pi^2}{3888\sqrt{3}}$ &  \small{Ry\u{s}kov and Baranovskii 1975 \cite{Ryskov}}\\
\bottomrule
\end{tabular}
\end{table}

Sphere packing and covering is not only an important subject in pure mathematics, but also has significant applications in many other subjects, such as in coding theory: \emph{the Shortest Vector Problem and the Closest Vector Problem, as significant problems in lattice based cryptography, are closely related to the densest lattice packing density and the thinnest lattice covering density of the ball}. For more details, we refer to \cite{Micciancio}.
However, with the further development of science and technology, it is impossible for scientists to stay in the existing post-quantum cryptosystem, and they will certainly study and develop cryptosystems with higher security (complexity).
And we also know that, in addition to the $n$-dimensional ball, the most interesting convex bodies are probably the $n$-dimensional regular polytopes that exist in every dimension: the cube, the simplex, and the cross-polytope. The $n$-cube is a tile, so its covering density is 1. Now we will focus on the remaining two.

Let $T_n$ denote an $n$-dimensional regular simplex and let $K$ denote a convex body. Our knowledge about simplices coverings is comparatively complete. In $\mathbb{E}^2$, we have
$$
\theta^t(T_2)=\theta^l(T_2)=\frac{3}{2},
$$
where the value of $\theta^l(T_2)$ was determined by I. F\'{a}ry \cite{Fary} in 1950, however the value of $\theta^t(T_2)$ was proved only in 2010 by J. Januszewski \cite{Januszewski}. It is even more surprising that, up to now some basic covering problems in the plane are still open (see \cite{Zong.2}). For example, we do not know yet if $\theta^t(K)=\theta^l(K)$ holds for all convex domains.
In $\mathbb{E}^3$, we have
\begin{align}
\frac{2^{16}+1}{2^{16}}\leq\theta^l(T_3)\leq\theta^t(T_3)\leq\frac{125}{63},\nonumber
\end{align}
where the upper bound was discovered by C. M. Fiduccia, R. W. Forcade and J. S. Zito \cite{Fiduccia}, R. Dougherty and V. Faber \cite{Dougherty} and R. Forcade and J. Lamoreaux \cite{Forcade} in 1990s, and the lower bound was achieved by F. Xue and C. Zong \cite{Xue} in 2018 by studying the volumes of generalized difference bodies. In 2006, J. H. Conway and S. Torquato \cite{Conway} obtained
\begin{align}
\theta^c(T_3)\leq\frac{9}{8}\nonumber
\end{align}
by constructing a particular tetrahedron covering.
Recently, Y. Li, M. Fu and Y. Zhang \cite{Li} proved
\begin{align}
\theta^t(T_3)\geq 1+1.227\times 10^{-3}.\nonumber
\end{align}
In $\mathbb{E}^n$, from the works of R. P. Bambah, H. S. M. Coxeter, H. Davenport, P. Erd\H{o}s, L. Few, G. L. Watson, and especially C. A. Rogers \cite{RogersB}, we know that
\begin{gather}
\theta^l(K)\leq n^{\log_2\log_e n+c}, \notag\\
\theta^t(K)\leq n\log n+n\log \log n+5n.\notag
\end{gather}
Since the 1960s, progress in covering is very limited (see \cite{Brass}). In 2018, F. Xue and C. Zong \cite{Xue} obtained
$$
\theta^l(T_n)\geq 1+\frac{1}{2^{3n+7}}.
$$
In 2021, O. Ordentlich, O. Regev and B. Weiss \cite{Ordentlich} improved Rogers' upper bound to
\begin{align}
\theta^l(K)\leq cn^2, \nonumber
\end{align}
where $c$ is a suitable positive constant. Recently, M. Fu, F. Xue and C. Zong \cite{Fu} proved that
$$
\theta^l(T_n)>1+\frac{1}{n}-\frac{1}{(n-1)2^{n-1}}.
$$

As the unit ball of the $l_1$ space, the $n$-dimensional cross-polytope $C_n$ has important applications in coding theory: a code in $Z_m^n$ having minimum Lee distance $d$ can be thought of as a packing of discrete cross-polytopes of radius $\lfloor\frac{d-1}{2}\rfloor$ in the torus, see \cite{Golomb}. Nevertheless, our knowledge about it is still very limited, both in terms of packing and covering.
The only known exact value is that the lattice packing density of $C_3$ (octahedron) is $\frac{18}{19}$, which was proved by Minkowski \cite{Minkowski} in 1904. It is obvious that $\theta^l(C_3)\leq\frac{9}{8}$ because the truncated octahedron is a parallelohedron (proved by Fedorov \cite{Fedorov} in 1885).
It is rather surprising that the exact value of $\theta^l(C_3)$ is still unknown, and nothing nontrivial is known about $\theta^t(C_3)$.

In this paper, we prove the following result:
\begin{thm}\label{density}
If $C_3+X$ is a translative covering of $\mathbb{E}^3$, then its density is at least $1+\frac{4}{6^{10}}$. In other words, we have
\begin{align}
\theta^t(C_3)\geq1+\frac{4}{6^{10}}>1+6.6\times10^{-8}.\nonumber
\end{align}
\end{thm}

\section{Preliminaries}
\quad \quad Define
\begin{align}\label{defineP}
P=conv\{&(8,0,0), (-8,0,0), (0,8,0), (0,-8,0),\\ \nonumber
& (0,0,8), (0,0,-8), (8,8,8),(-8,-8,-8)\}, \nonumber
\end{align}
which is a parallelohedron, i.e., a tile for $\mathbb{E}^3$.
Since $\theta^t(K)$ is invariant under non-singular linear transformations on $K$, we will work on the regular octahedron $C_3$ with vertices $(2,0,0)$, $(0, 2, 0)$, $(0, 0, 2)$, $(-2, 0, 0)$, $(0,-2,0)$, $(0,0,-2)$, edge length $2\sqrt{2}$ and volume ${32\over3}$. It is obvious that
\begin{align}\label{vol(P)}
4C_3\subset P \ \textrm{and}\ vol(P)=1024,
\end{align}
where $vol(\cdot)$ denotes the volume of a given convex body.

Define
\begin{align}
D(K)=\{\mathbf{x}-\mathbf{y}: \mathbf{x}, \mathbf{y}\in K\}. \nonumber
\end{align}
Usually, we call $D(K)$ the \emph{difference body} of $K$. Clearly $D(K)$ is a centrally symmetric convex body centered at the origin $\mathbf{o}$.
\begin{rem}
When $K$ is a centrally symmetric convex body centered at the origin $\mathbf{o}$, then $D(K)=2K$. Obviously, $D(C_3)=2C_3$.
\end{rem}

\begin{lem}\label{5T}
If $C_3\cap(C_3+\mathbf{x})\neq\varnothing$, then $C_3+\mathbf{x}\subset 3C_3$.
\end{lem}
\begin{proof}

Since $C_3\cap(C_3+\mathbf{x})\neq\varnothing$, without loss of generality, suppose that $\mathbf{z}\in C_3\cap(C_3+\mathbf{x})$. Then, we have $\mathbf{z}\in C_3$ and $-\mathbf{x}\in C_3-\mathbf{z}$. By the symmetry of $C_3$, we have
\begin{align}
\mathbf{x}\in C_3+\mathbf{z}\subset D(C_3).\nonumber
\end{align}
Therefore, $C_3+\mathbf{x}\subset C_3+D(C_3)=3C_3$ if $C_3\cap(C_3+\mathbf{x})\neq\varnothing$, the lemma is proved.
\end{proof}

\begin{rem}\label{=D}
In fact, 
By Lemma \ref{5T} and (\ref{vol(P)}), when $\mathbf{o}\in C_3+\mathbf{x}$ and $(C_3+\mathbf{x})\cap(C_3+\mathbf{y})\neq\varnothing$,
\begin{align}\label{relation}
\bigcup(C_3+\mathbf{y})\subset(3C_3+{\bf x})\subset 4C_3 \subset P,
\end{align}
\end{rem}

Let
\begin{align}
H_{z_0}=\{(x, y, z): z=z_0, x, y\in \mathbb{R}\}\nonumber
 \end{align}
be the $2$-dimensional subspace of $\mathbb{E}^3$ which is perpendicular to $z$-axis. Thus, by a routine argument we can deduce that
\begin{lem}\label{homothetic}
Suppose $\mathbf{x}=(x, y, z)$, then we have

(1) The $2$-dimensional cross section $H_{z_0}\cap (C_3+\mathbf{x})\neq \varnothing$ if and only if $z_0\in [-2+z, 2+z]$;

(2) $H_{z_0}\cap (C_3+\mathbf{x})$ is a square centered at $(x, y, z_0)$.
\end{lem}

\section{Translative coverings of octahedron}
\quad \quad To study $\theta^t(C_3)$, the most natural approach is localization. Assume that $X$ is a discrete set of points in $\mathbb{E}^3$ such that $C_3+X$ is a translative covering of $\mathbb{E}^3$. Let $P$ be the parallelohedron defined in (\ref{defineP}). Then define
\begin{align}\label{eq:CXP}
\theta(C_3,X,P)=\frac{\sum\limits_{\mathbf{x}\in X}vol(P\cap(C_3+\mathbf{x}))}{vol(P)}.
\end{align}
Let $\mathfrak{X}$ denote the family of all such sets $X$. We call
\begin{align}\label{eq:CP}
\theta(C_3,P)=\min_{X\in\mathfrak{X}}\theta(C_3,X,P)
\end{align}
the covering density of $C_3$ for $P$. Since $P$ is a parallelohedron, clearly,
\begin{align}\label{eq:C}
\theta^t(C_3)\geq\theta(C_3,P).
\end{align}

\begin{proof}[\textbf{Proof of Theorem \ref{density}}]
Since $C_3+X$ is a translative covering of $\mathbb{E}^3$, there must exist a subset $X'\subset X$ such that $C_3+X'$ is a covering of $P$. Assume that $\mathbf{o}\in C_3+\mathbf{x}_{m+1}$, where $\mathbf{x}_{m+1}\in X'$. Then denote 
\begin{align}
X''=\{\mathbf{x}_i\in X': (C_3+\mathbf{x}_{m+1})\cap (C_3+\mathbf{x}_i)\neq \varnothing\}.\nonumber
\end{align}
Without loss of generality, we suppose that 
\begin{align}
X''=\{\mathbf{x}_1, \ldots, \mathbf{x}_{m+1}\}.\nonumber
\end{align}
By (\ref{relation}), we have
\begin{align}\label{contain}
\bigcup\limits_{i=1}^{m+1}(C_3+\mathbf{x}_i)\subset 3C_3+ \mathbf{x}_{m+1}\subset 4C_3 \subset P.
\end{align}
Now, we estimate $\theta(C_3,X,P)$ by considering two cases.\\

\textbf{Case 1. $m\geq 27$.} By (\ref{vol(P)}), (\ref{eq:CXP}) and (\ref{contain}),
\begin{align}
\theta(C_3,X,P)&=\frac{\sum\limits_{\mathbf{x}\in X}vol\left(P\cap(C_3+\mathbf{x})\right)}{vol(P)} \nonumber \\
&\geq\frac{vol(P\setminus (3C_3+\mathbf{x}_{m+1}))+\sum\limits_{\mathbf{x}\in X}vol\left((3C_3+\mathbf{x}_{m+1})\cap(C_3+\mathbf{x})\right)}{vol(P)} \nonumber \\
&\geq1-\frac{9}{32}+\frac{\sum\limits_{i=1}^{m+1}vol\left((3C_3+\mathbf{x}_{m+1})\cap(C_3+\mathbf{x}_i)\right)}{vol(P)} \nonumber \\
&=\frac{23}{32}+(m+1)\cdot\frac{vol(C_3)}{vol(P)} \nonumber \\
&\geq1+ \frac{1}{96}. \nonumber
\end{align}\\

\textbf{Case 2. $m\leq 26$.} Let $\partial(\cdot)$ denote the boundary of a convex body. Then
\begin{align}\label{boundary}
\partial(C_3+\mathbf{x}_{m+1})=\bigcup_{i=1}^{m}\left(\partial(C_3+\mathbf{x}_{m+1})\cap(C_3+\mathbf{x}_i)\right).
\end{align}
Assume that $R_{z_0}^{(i)}=H_{z_0}\cap (C_3+\mathbf{x}_i), i=1, \ldots, m+1$, then by Lemma \ref{homothetic}, $R_{z_0}^{(i)}$ are squares with edges length less than or equal to $2\sqrt{2}$ (including zero).

Firstly, let $l_{z_0}^{(i)}$ denote the length of the edges of $R_{z_0}^{(i)}$, where $1\leq i\leq m+1$. 
It can be checked that when $l_{z_0}^{(i)}\leq \frac{\sqrt{2}}{27}$ or $l_{z_0}^{(i)}\geq \frac{53\sqrt{2}}{27}$, the length of interval where $z_0$ can satisfy $R_{z_0}^{(i)}$ existing is $\frac{4}{27}$. Combining $m+1\leq 27$, we have
\begin{align}
(m+1)\times \frac{4}{27}\leq 4.\nonumber
\end{align}
Let $\mathbf{x}_i=(x_i, y_i, z_i)$, where $1\leq i\leq m+1$. Therefore, there exists a $z_0\in(-2+z_{m+1}, 2+z_{m+1})$ such that  $R_{z_0}^{(m+1)}$ is a nonempty set and $R_{z_0}^{(i)}$ are squares with  $\frac{\sqrt{2}}{27}\leq l_{z_0}^{(i)}\leq \frac{53\sqrt{2}}{27}$, where $1\leq i\leq m+1$.

Next, by (\ref{boundary}), the four edges of $R_{z_0}^{(m+1)}$ are covered by $\bigcup_{k=1}^{m} R_{z_0}^{(k)}$. Let $A, B, C, D$ be the four vertices of $R_{z_0}^{(m+1)}$ and let $E, F, G, H$  be the four vertices of $R_{z_0}^{(k)}$. Based on the number of vertices of $R_{z_0}^{(m+1)}$ covered by $R_{z_0}^{(k)}$, there are three possible covering configurations as shown in Fig. \ref{fig 1}.
\begin{figure}
\centering
\includegraphics[height=3.3cm]{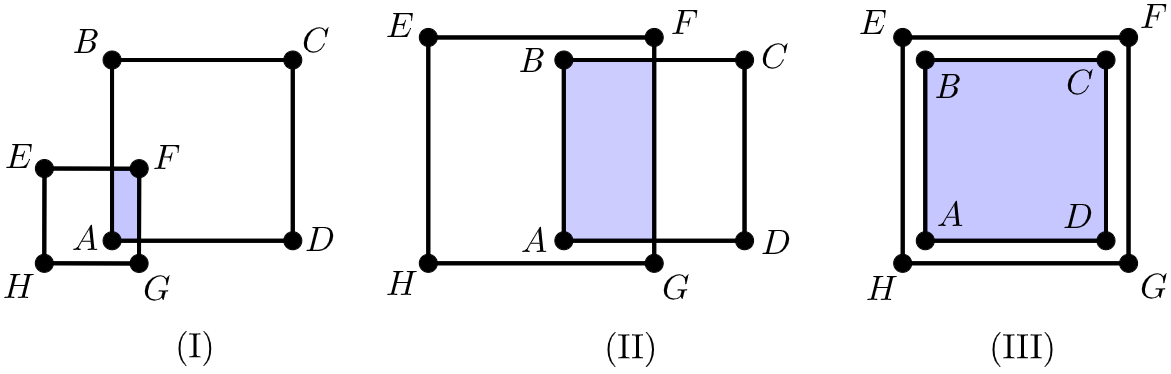}
\caption{Possible relative situation between $R_{z_0}^{(m+1)}$ and $R_{z_0}^{(i)}$.}
\label{fig 1}
\end{figure}

In configuration (I), let $a$ be the distance between $AD$ and $EF$ and let $b$ be the distance between $AB$ and $FG$. Without loss of generality, we assume that
\begin{align}
0\leq b \leq a \leq l_{z_0}^{(i)} \leq l_{z_0}^{(m+1)}.\nonumber
\end{align}
In configuration (II), let $c$ be the distance between $BC$ and $EF$, let $d$ be the distance between $AD$ and $HG$, let $e$ be the distance between $AB$ and $FG$. Without loss of generality, we assume that
\begin{align}\nonumber
 c\leq d.
\end{align}
In configuration (III), denote the distance between $AD$ and $HG$, $AB$ and $HE$, $CD$ and $FG$, $BC$ and $EF$ with $g, h, r, s$, respectively. Without loss of generality, we assume that
\begin{align}
g\leq h \leq r \leq s.\nonumber
\end{align}
Let $\mathbf{x}_i=(x_i, y_i, z_i)$, $1\leq i\leq m+1$. For a fixed $z_0\in (-2+z_{m+1}, 2+z_{m+1})$, define

\begin{figure}
\centering
\includegraphics[height=4.8cm]{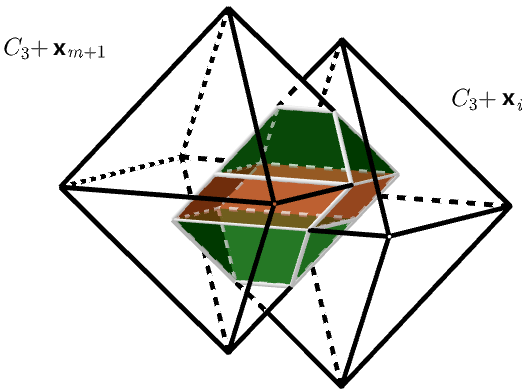}
\caption{Configuration (I) in subcase 1.}
\label{fig 2}
\end{figure}

\begin{equation}
\delta(R_{z_0}^{(i)})=\left\{
\begin{array}{rcl}
 +1,& &\text{$z_0\in [-2+z_i,z_i]$},\\\nonumber
 -1,& &\text{$z_0\in (z_i,2+z_i]$.}\nonumber
\end{array}\right.
\end{equation}
We consider two subcases.\\
\textbf{Subcase 1. $\delta(R_{z_0}^{(m+1)})\cdot\delta(R_{z_0}^{(i)})=1$.}

Without loss of generality, assume that $\delta(R_{z_0}^{(m+1)})=1, \delta(R_{z_0}^{(i)})=1$. Take configuration (I) as an example, as shown in Fig. \ref{fig 2}, we have
\begin{align}\label{eq:con1}
vol((C_3+{\mathbf{x}_{m+1}})\cap(C_3+{\mathbf{x}_i}))
&=\int_{-2+z_{m+1}}^{2+z_{m+1}}S(R_{z}^{(m+1)}\cap R_{z}^{(i)})dz \nonumber \\
&\geq\int_{z_0-{b\over\sqrt{2}}}^{z_0+{{2\sqrt{2}-l_{z_0}^{(m+1)}}\over\sqrt{2}}}S(R_{z}^{(m+1)}\cap R_{z}^{(i)})dz \nonumber\\
&+\int_{z_0+{{2\sqrt{2}-l_{z_0}^{(i)}}\over\sqrt{2}}}^{z_0+{{4\sqrt{2}-l_{z_0}^{(m+1)}-l_{z_0}^{(i)}+b}\over\sqrt{2}}}S(R_{z}^{(m+1)}\cap R_{z}^{(i)})dz \nonumber \\
&\geq 2\times{1\over3\sqrt{2}}(b+2\sqrt{2}-l_{z_0}^{(m+1)})^{3} \nonumber \\
&\geq \frac{\sqrt{2}}{3}\left(2\sqrt{2}-\frac{53\sqrt{2}}{27}\right)^{3} \nonumber \\
&=\frac{4}{3^{10}},
\end{align}
where $S(\cdot)$ denotes the area of a two dimensional convex set.

For configuration (II), we have
\begin{align}
vol((C_3+{\mathbf{x}_{m+1}})\cap(C_3+{\mathbf{x}_i}))
&=\int_{-2+z_{m+1}}^{2+z_{m+1}}S(R_{z}^{(m+1)}\cap R_{z}^{(i)})dz \nonumber \\
&\geq \int_{z_0-{{e\over\sqrt{2}}}}^{z_0+{{{2\sqrt{2}-l_{z_0}^{(i)}}\over\sqrt{2}}}}S(R_{z}^{(m+1)}\cap R_{z}^{(i)})dz \nonumber \\
&+ \int_{z_0+{{{2\sqrt{2}-l_{z_0}^{(m+1)}}\over\sqrt{2}}}}^{z_0+{{{4\sqrt{2}-l_{z_0}^{(m+1)}-l_{z_0}^{(i)}+e}\over\sqrt{2}}}}S(R_{z}^{(m+1)}\cap R_{z}^{(i)})dz \nonumber \\
&\geq 2\times{1\over3\sqrt{2}}(2\sqrt{2}-l_{z_0}^{(i)}+e)^{2}(2\sqrt{2}-l_{z_0}^{(i)}+l_{z_0}^{(m+1)}) \nonumber \\
&\geq  {\sqrt{2}\over3}({1\over27}\sqrt{2})^{3} \nonumber \\
&= {4\over3^{10}}. \nonumber \\
\end{align}

For configuration (III), we have
\begin{align}
vol((C_3+{\mathbf{x}_{m+1}})\cap(C_3+{\mathbf{x}_i}))
&=\int_{-2+z_{m+1}}^{2+z_{m+1}}S(R_{z}^{(m+1)}\cap R_{z}^{(i)})dz \nonumber \\
&\geq\int_{z_0-{{l_{z_0}^{(m+1)}\over\sqrt{2}}}}^{z_0+{{{2\sqrt{2}-l_{z_0}^{(i)}+g}\over\sqrt{2}}}}S(R_{z}^{(m+1)}\cap R_{z}^{(i)})dz \nonumber \\
&+\int_{z_0+{{{2\sqrt{2}-l_{z_0}^{(i)}+s}\over\sqrt{2}}}}^{z_0+{{{4\sqrt{2}-l_{z_0}^{(i)}}\over\sqrt{2}}}}S(R_{z}^{(m+1)}\cap R_{z}^{(i)})dz \nonumber \\
&=  2\times{1\over3\sqrt{2}}(g+2\sqrt{2}-l_{z_0}^{(i)}+l_{z_0}^{(m+1)})^{3}\nonumber \\
&\geq  \frac{\sqrt{2}}{3}(2\sqrt{2}-\frac{53\sqrt{2}}{27})^{3}\nonumber \\
&=\frac{4}{3^{10}}.\nonumber \\
\end{align}

\begin{rem}
The configurations (I), (II), and (III) in Fig. \ref{fig 1} indicate that $R_{z_0}^{(k)}$ covers 1, 2, and 3 (4) vertices of $R_{z_0}^{(m+1)}$, respectively.  In fact, the configuration is equivalent to (II) when $R_{z_0}^{(k)}$ does not cover any vertex of $R_{z_0}^{(m+1)}$.
\end{rem}
\textbf{Subcase 2. $\delta(R_{z_0}^{(m+1)})\cdot\delta(R_{z_0}^{(i)})=-1$.}

Based on the number of vertices of $R_{z_0}^{(m+1)}$ covered by $R_{z_0}^{(k)}$, there are three possible covering configurations as shown in Fig. \ref{fig 1}.

Without loss of generality, assume that $\delta(R_{z_0}^{(m+1)})=1$, $\delta(R_{z_0}^{(i)})=-1$. In configuration (I) and (II), by the covering structure, it is obvious to see that the point of $AD$ intersected with $FG$ must be covered by another $R_{z_0}^{(j)}$, where $j\in\{1,2,\ldots,m\}\setminus \{i\}$. Since one of $\delta(R_{z_0}^{(m+1)})$ and $\delta(R_{z_0}^{(i)})$ must have the same sign with $\delta(R_{z_0}^{(j)})$, by the result of Subcase 1, we have
\begin{align}
vol((C_3+{\mathbf{x}_{m+1}})\cap (C_3+{\mathbf{x}_j}))+vol((C_3+{\mathbf{x}_i})\cap (C_3+{\mathbf{x}_j}))\geq {4\over3^{10}}.
\end{align}
For configuration (III), we have
\begin{align}\label{eq:con3}
vol((C_3+{\mathbf{x}_{m+1}})\cap(C_3+{\mathbf{x}_i}))
&=\int_{-2+z_{m+1}}^{2+z_{m+1}}S(R_{z}^{(m+1)}\cap R_{z}^{(i)})dz \nonumber \\
&\geq \int_{z_0-{{l_{z_0}^{(m+1)}\over\sqrt{2}}}}^{z_0+{{g\over\sqrt{2}}}}S(R_{z}^{(m+1)}\cap R_{z}^{(i)})dz \nonumber \\
&+\int_{z_0+{{s\over\sqrt{2}}}}^{z_0+{{l_{z_0}^{(i)}\over\sqrt{2}}}}S(R_{z}^{(m+1)}\cap R_{z}^{(i)})dz\nonumber \\
&= 2\times{1\over3\sqrt{2}}(g+l_{z_0}^{(m+1)})^{3}\nonumber \\
&\geq {\sqrt{2}\over3}({1\over27}\sqrt{2})^{3}\nonumber \\
&={4\over3^{10}}.\nonumber \\
\end{align}

Therefore, by the conclusions of (\ref{eq:CXP}) and (\ref{eq:con1})-(\ref{eq:con3}), we have
\begin{align}
\theta(C_3,X,P)&=\frac{\sum\limits_{\mathbf{x}\in X}vol\left(P\cap(C_3+\mathbf{x})\right)}{vol(P)}\nonumber \\
&=1+\frac{\sum\limits_{\mathbf{x}_i,\mathbf{x}_j\in X',i<j}vol\left(P\cap(C_3+\mathbf{x}_i)\cap(C_3+\mathbf{x}_j)\right)}{vol(P)} \nonumber \\
&\geq 1+\frac{\sum\limits_{1\leq i<j\leq m+1}vol\left((C_3+\mathbf{x}_i)\cap(C_3+\mathbf{x}_j)\right)}{vol(P)}\nonumber \\
&\geq1+\frac{1}{1024}\cdot \frac{4}{3^{10}}\nonumber \\
&= 1+\frac{4}{6^{10}}. \nonumber
\end{align}

As a conclusion of these two cases, we have
\begin{align}
\theta(C_3,X,P)\geq 1+\frac{4}{6^{10}}>1+6.6\times10^{-8}. \nonumber
\end{align}
According to equations (\ref{eq:CXP}), (\ref{eq:CP}) and (\ref{eq:C}), we obtain
\begin{align}
\theta^t(C_3)\geq \theta(C_3,P)=\min_{X\in\mathfrak{X}}\theta(C_3,X,P)\geq 1+\frac{4}{6^{10}}>1+6.6\times10^{-8}, \nonumber
\end{align}
and the theorem is proved.
\end{proof}

\begin{rem}
By covering the structure with asymptotic $\delta^l(C_3)=\frac{18}{19}$ instead of $\frac{vol(C_3)}{vol(P)}=\frac{2}{3}$, we can slightly improve the lower bound in Theorem \ref{density}. However, since the improvement is not essential, its proof is not include here.
\end{rem}
\section*{Acknowledgement}
\quad \quad This work is supported by the National Natural Science Foundation of China (NSFC12226006, NSFC11921001, NSFC11801410 and NSFC11971346) and the Natural Key Research and Development Program of China (2018YFA0704701) and the China Scholarship Council. The authors are grateful to Professor C. Zong for his supervision and discussion.

\end{document}